\documentclass[10pt]{amsart}
\usepackage{amsmath}
\usepackage[usenames,dvipsnames]{color}
\usepackage{parskip}
\usepackage{amsfonts}
\usepackage{amscd}
\usepackage[centertags]{amsmath}
\usepackage{amssymb}
\usepackage{stmaryrd}
\usepackage[all,cmtip]{xy}
\usepackage[english]{babel}
\usepackage{tabularx}
\usepackage{amsxtra}
\usepackage{euscript}
\usepackage[T1]{fontenc}
\usepackage{doc, exscale, fontenc, latexsym, syntonly}
\usepackage{amsfonts}
\usepackage{amsthm}
\usepackage{graphicx}
\usepackage{mciteplus}
\usepackage{cite}

\newtheorem{theorem}{Theorem}[section]
\newtheorem{corollary}[theorem]{Corollary}

\newtheorem{proposition}[theorem]{Proposition}
\newtheorem{remark}[theorem]{Remark}
\theoremstyle{definition}
\newtheorem{definition}[theorem]{Definition}
\newtheorem{example}[theorem]{Example}
\theoremstyle{remark}

\numberwithin{figure}{section}
\numberwithin{table}{section}

\newcommand*\acknowledgment[1]{%
	\begingroup\noindent
	\rightskip\leftskip
	\begin{flushleft}\textbf{\large Acknowledgment.}\, #1%
		\par\vspace*{1mm}\end{flushleft}\endgroup}
\begin{document}

\title[Topological Complexity Related To Multi-Valued Functions]{Topological Complexity Related To Multi-Valued Functions}

\author{MEL\.{I}H \.{I}S}
\date{\today}

\address{\textsc{Melih Is,}
Ege University\\
Faculty of Sciences\\
Department of Mathematics\\
Izmir, Turkiye}
\email{melih.is@ege.edu.tr}



\begin{abstract}
 In this paper, we deal with the robot motion planning problem in multi-valued function theory. We first enrich the multi-homotopy studies by introducing a multi-homotopy lifting property and a multi-fibration. Then we compute both a topological multi-complexity and a Lusternik-Schnirelmann multi-category of a space or a multi-fibration.
\end{abstract}

\maketitle

\section{Introduction}
\label{intro}

\quad In recent decades, the field of robotics has witnessed remarkable advancements, leading to the integration of autonomous robots into various aspects of our daily lives, from manufacturing and healthcare to exploration and transportation. A fundamental challenge in this domain is the efficient and reliable planning of robot motion through complex and dynamic environments. As robots become more prevalent and their tasks increasingly intricate, the need for innovative solutions to address the motion planning problem has grown exponentially. Traditional methods in robot motion planning have predominantly relied on single-valued functions, where the goal is to find a collision-free path for a robot from its initial configuration to a desired final configuration. While these techniques have yielded substantial progress, they often encounter limitations in handling complex, high-dimensional spaces and non-deterministic environments. The quest for more versatile and adaptable approaches has driven researchers to explore new avenues, and among them, the integration of algebraic topology methods has shown great promise.

\quad Robot motion planning problem has been one of the main application areas of algebraic topology since Farber introduced the notion of topological complexity number of a path-connected space, denoted by TC$(X)$\cite{Farber:2003}. Later, this investigation enriches with the help of new concepts such as higher topological complexity number of a path-connected space\cite{Rudyak:2010} (denoted by TC$_{n}(X)$ for $n>1$), topological complexity numbers of a continuous surjection\cite{Pavesic:2019,MelihKaraca:2022} (denoted by TC$(f)$ or TC$_{n}(f)$ for a a continuous surjection), topological complexity numbers in digital topology\cite{KaracaMelih:2018,MelihKaraca:2020} (denoted by TC$(X,\kappa)$ or TC$_{n}(X,\kappa)$ for an adjacency relation $\kappa$ on a digitally connected digital image $X$), and topological complexity numbers on proximity spaces\cite{MelihKaraca:2023} (denoted by TC$(X,\delta)$ for a proximally connected space $(X,\delta)$). As a next step, it is important to examine the topological complexity computations on multi-valued functions, which is the main subject of this study, as it creates a new area of discussion.

\quad Multi-valued functions, a multifaceted concept with a rich historical heritage in mathematics, represent a profound departure from the traditional paradigm of single-valued functions. Their roots can be traced back to the pioneering work of 19th-century mathematicians particularly within the domains of complex analysis and algebraic geometry. Multi-valued functions are observed in the study of some important topological properties\cite{Borges:1967,Smithson:1965}, and in very broad fields of study, especially homotopy theory\cite{Strother:1955}. One can find recent studies on certain concepts of homotopy theory, such as homotopy extension property, fundamental group constructions, or covering spaces, using these functions in \cite{KarDenizTemiz:2021}, \cite{KaracaOzkan:2023}, \cite{TemizelKaraca:2022}, and \cite{TemizelKaraca:2023}. 

\quad As in single-valued functions, the primary goal of this work is to introduce topological complexity numbers for multi-valued functions by combining the ideas of homotopic distance and Schwarz genus. First, we examine certain concepts that guide us on how to handle the basic properties of multi-valued functions in Section \ref{sec:1}. Then, in Section \ref{sec:2}, we introduce the multi-homotopy lifting property and the well-known concept of fibration in a multi-valued setting. The purpose of this is to show that the path fibration of single-valued functions does not lose its fibration property when multi-valued functions are considered. Section \ref{sec:3} is dedicated to the concepts of $m-$Schwarz genus, topological $m-$complexity, $m-$homotopic distance, and important properties and examples of these concepts. After that, we generalize this examination to the problem of computing the topological $m-$complexity numbers of $m-$fibrations in Section \ref{sec:4}. The next section considers the well-known concept of the Lusternik-Schnirelmann category (simply denoted by LS or cat) for both a space and a function in terms of multi-valued settings. See \cite{CorneaLupOpTan:2003} to have detailed information for the category LS concerning single-valued functions. The last section is devoted to discussing the properties obtained in this study and what can be done in future studies of topological multi-complexity and its related invariants.

\section{Preliminaries}
\label{sec:1}

\quad In this section, we introduce some previously presented notions about multi-valued functions. To avoid confusion, throughout the article, we represent sets with capital letters such as $X$, $A$, and $F$, and functions with Greek letters such as $\alpha$, $\eta$, and $\mu$. 

\quad A \textit{multi-valued function} $\alpha$ from a topological space $A$ to another topological space $B$, denoted by $\alpha : A \rightrightarrows B$, is a function that assigns any point of $A$ to the nonempty subset of $B$\cite{Anisiu:1981}. For simplicity, we prefer to use ''$m-$'' instead of the words ''multi'' and ''multi-valued'' throughout the rest of the article. For example, we write $m-$function instead of multi-valued function. Consider an $m-$function $\alpha : A \rightrightarrows B$ and any open subset $V$ of $B$. Then \textit{upper inverse} $\alpha^{+}(V)$ and \textit{lower inverse} $\alpha^{-}(V)$ of $\alpha$ are given by
\begin{eqnarray*}
	\alpha^{+}(V) = \{a \in A : F(a) \subset V\} \hspace*{0.5cm} \text{and} \hspace*{0.5cm} \alpha^{-}(V) = \{a \in A : F(a) \cap V \neq \emptyset\},
\end{eqnarray*}
respectively\cite{Borges:1967}. Using the lower and the upper inverse, the continuity of $m-$functions can be defined as follows\cite{Borges:1967}. If $\alpha^{+}(V) \subset A$ is open for every open subset $V \subset B$, then $\alpha$ is called \textit{upper semicontinuous}. If $\alpha^{-}(V) \subset A$ is open for every open subset $V \subset B$, then $\alpha$ is called \textit{lower semicontinuous}. $\alpha$ is said to be \textit{$m-$continuous} provided that $\alpha$ is both upper and lower semicontinuous. 

\quad Given a subset $C$ of $A$, the \textit{$m-$inclusion} map is a function $$\alpha : Y \rightrightarrows A$$ with $\alpha(y) = \{y\}$ for any point $y \in Y$\cite{KarDenizTemiz:2021}. Note that an $m-$inclusion function is $m-$continuous\cite{KarDenizTemiz:2021}. Given two spaces $A$ and $B$, the \textit{$i-$th $m-$projection} map is a function $$\alpha : A_{1} \times A_{2} \times \cdots \times A_{n} \rightrightarrows A_{i}$$ with $\alpha(a_{1},a_{2},\cdots,a_{n}) = \{a_{i}\}$ for all $(a_{1},a_{2},\cdots,a_{n}) \in A_{1} \times A_{2} \times \cdots \times A_{n}$ and each $i \in \{1,2,\cdots,n\}$. An \textit{$m-$constant} map $$\alpha : A \rightrightarrows B$$ is a function with $\alpha(a) = B_{0}$ for all $a \in A$\cite{KarDenizTemiz:2021}. An \textit{$m-$identical} map $$1_{A} : A \rightrightarrows A$$ on a space $A$ is a function with $1_{A}(a) = \{a\}$ for all $a \in A$\cite{Anisiu:1981}. Given an $m-$function $\alpha : A \rightrightarrows B$, it is said to be an \textit{$m-$section} provided that there is an $m-$function $\beta : B \rightrightarrows A$ with the property $$\{a\} = \displaystyle \bigcap_{b \in \alpha(a)}\beta(b)$$ for all $a \in A$\cite{KaracaOzkan:2023}.

\quad The \textit{inverse of an $m-$function} $\alpha : A \rightrightarrows B$ is a function $$\alpha^{-1} : B \rightrightarrows A$$ with $\alpha^{-1}(D) = \{a \in A : \alpha(a) \subset D\}$ for each nonempty subset $D \subset B$\cite{Borges:1967}. An $m-$function $\alpha : A \rightrightarrows B$ is said to be \textit{one-to-one} provided that for any distinct $a_{1}$, $a_{2} \in A$, the intersection $\alpha(a_{1}) \cap \alpha(a_{2})$ is an empty set\cite{Anisiu:1981}. An $m-$function $\alpha : A \rightrightarrows B$ is said to be \textit{surjective} provided that the range of $\alpha$ is $B$\cite{Anisiu:1981}. Moreover, a surjective and one-to one $m-$function $\alpha$ is called \textit{$m-$homeomorphism} provided that both $\alpha$ and the inverse $\alpha^{-1} : B \rightrightarrows A$ of $\alpha$ is $m-$continuous\cite{KarDenizTemiz:2021}.

\quad An $m-$continuous function $$\alpha : I \rightrightarrows A$$ is called an \textit{$m-$path} on $A$, where $I = [0,1]$\cite{Strother:1955}. Let $A$ be a space. Then it is called \textit{$m-$pathwise connected} provided that there is an $m-$continuous function $\alpha : I \rightrightarrows A$ with the properties $\alpha(0) = A_{0}$ and $\alpha(1) = A_{1}$ for any closed subsets $A_{0}$, $A_{1} \subset A$\cite{Strother:1955}. Given two $m-$functions $\alpha$, $\beta : A \rightrightarrows B$, they are said to be \textit{$m-$homotopic} provided that there is an $m-$continuous function $$\gamma : A \times I \rightrightarrows B$$ with the properties $\gamma(a,0) = \alpha(a)$ and $\gamma(a,1) = \beta(a)$\cite{Strother:1955}. An $m-$continuous function $\alpha : A \rightrightarrows B$ is an \textit{$m-$homotopy equivalence} if there is an $m-$continuous function $\beta : B \rightrightarrows A$ with the properties that $\beta \circ \alpha$ is $m-$homotopic to the $m-$identical map on $A$ and $\alpha \circ \beta$ is $m-$homotopic to $m-$identical map on $B$\cite{KarDenizTemiz:2021}. Also, $A$ and $B$ \textit{have the same $m-$homotopy type} provided that there is an $m-$homotopy equivalence. A topological space $A$ is said to be an \textit{$m-$contractible} provided that the $m-$identical map and the $m-$constant map are $m-$homotopic on $A$\cite{KarDenizTemiz:2021}. $\alpha : A \rightrightarrows B$ is said to be \textit{null $m-$homotopic} if $\alpha$ and an $m-$constant map are $m-$homotopic\cite{KarDenizTemiz:2021}. This shows that the $m-$contractibility can also be stated as follows. If the $m-$identical map on $A$ is null $m-$homotopic, then we say that $A$ is $m-$contractible.

\quad Continuous function spaces on $m-$functions are widely studied in \cite{GuptaSarma:2017}. Given two topological spaces $A$ and $B$, an $m-$continuous \textit{$m-$function space}, denoted by $B^{A}_{m}$, is the collection of all $m-$continuous $m-$functions from $A$ to $B$. Let $A$ and $B$ be two topological spaces. Then a function $$\text{Ev} : B^{A}_{m} \times A \rightrightarrows B$$ with Ev$(\alpha,a) = \alpha(a)$ is said to be an \textit{$m-$evaluation map}. This function is clearly $m-$continuous since $\alpha \subset B^{A}_{m}$ is $m-$continuous.   

\quad Homotopy extension property for $m-$functions is defined in \cite{KaracaOzkan:2023} as follows. Let $A$ be a space and $C$ a subset of $A$. Then the pair $(A,C)$ has an \textit{$m-$homotopy extension property} provided that for any $m-$continuous function $\beta : A \rightrightarrows b$ and any $m-$homotopy $\beta^{'} : C \times I \rightrightarrows B$ with $\beta(a) = \beta^{'}(a,0)$ for all $a \in C$, there is an $m-$homotopy $$\alpha : A \times I \rightrightarrows B$$ for which $\alpha(a,0) = \beta(a)$ and $\alpha|_{C \times I} = \beta^{'}$. 

\section{Multi-Fibrations}
\label{sec:2}

\quad An $\epsilon-\delta$ definition of an $m-$homotopy lifting property with respect to a topological space is given in \cite{GiraldoSanjurjo:2001}. To compute topological complexity numbers in $m-$functions, we need a more technical definition of the $m-$homotopy lifting property (or an $m-$fibration) in terms of comprehensive homotopy theory.

\begin{definition}\label{def3}
	Let $\rho : A \rightrightarrows B$ be an $m-$continuous function. Then $\rho$ has the $m-$homotopy lifting property (simply $m-$hlp) with respect to a space $W$ provided that there exists a filler $\eta : W \times I \rightrightarrows A$ for all commutative diagrams
	\begin{displaymath}
		\xymatrix{
			W \ar[r]^{\alpha} \ar[d]_{\nu_{0}} &
			A \ar[d]^{\rho} \\
			W \times I \ar[r]_{\beta}  \ar@{.>}[ur]^{\eta} & B,}
	\end{displaymath}
    where $\nu_{0} : W \rightrightarrows W \times I$ is the $m-$inclusion defined by $\nu_{0}(w) = \{(w,0)\}$ for all $w \in W$.
\end{definition}

\quad The function $\rho$ is called a multi-fibration (simply $m-$fibration) if it has the $m-$hlp with respect to any space.

\begin{example}
	If $\rho : A \rightrightarrows B$ is constant, then it is an $m-$fibration. Indeed, the filler $\eta$ can be defined as $\eta(w,t) = \alpha(w)$ for any $w \in W$.
\end{example}

\begin{example}
	The $m-$projection $\Pi_{i} : A_{1} \times A_{2} \rightrightarrows A_{i}$ is an $m-$fibration with each $i \in \{1,2\}$. For $i = 1$, the filler $\eta : W \times I \rightrightarrows A_{1} \times A_{2}$ can be defined by \[\eta(w,t) = (\beta(w,t),\Pi_{2} \circ \alpha(w)),\] where $\Pi_{2} : A_{1} \times A_{2} \rightrightarrows A_{2}$ is the second $m-$projection. Also, the case $i = 2$ admits a similar construction, i.e.,
	there exists a filler $\eta : W \times I \rightrightarrows A_{1} \times A_{2}$ given by \[\eta(w,t) = (\Pi_{1} \circ \alpha(w),\beta(w,t))\] for the first $m-$projection $\Pi_{1} : A_{1} \times A_{2} \rightrightarrows A_{1}$.
\end{example}

\begin{proposition}\label{prop2}
	An $m-$homeomorphism is an $m-$fibration.
\end{proposition}

\begin{proof}
	Let $\rho : A \rightrightarrows B$ be an $m-$homeomorphism, $\alpha : W \rightrightarrows A$ any $m-$function, $\beta : W \times I \rightrightarrows B$ any $m-$homotopy, and $\nu_{0} : W \rightrightarrows W \times I$ the $m-$inclusion. Then the $m-$function $\eta : W \times I \rightrightarrows A$ defined by $\eta = \alpha^{-1} \circ \beta$ is the desired filler with the property that $\rho \circ \eta = \beta$ and $\eta \circ \nu_{0} = \alpha$.
\end{proof}

\quad Let $\rho : A \rightrightarrows B$ be an $m-$fibration and $\gamma : B^{'} \rightrightarrows B$ an $m-$continuous function. Then the $m-$pullback is given by $\gamma^{\ast}A = \{(b^{'},a) : \gamma(b^{'}) = \rho(a)\} \subset B^{'} \times A$ (see the following diagram, where $\Pi_{1} : \gamma^{\ast}A \rightrightarrows B^{'}$ is the first $m-$projection map and $\Pi_{2} : \gamma^{\ast}A \rightrightarrows A$ is the second $m-$projection map.).
$$\xymatrix{
	\gamma^{\ast}A \ar@<1ex>[r]^{\Pi_{2}} \ar@<-1ex>[r] \ar@<1ex>[d] \ar@<-1ex>[d]_{\Pi_{1}} &
	A \ar@<-1ex>[d] \ar@<1ex>[d]^{\rho} \\
	B^{'} \ar@<-1ex>[r]_{\gamma} \ar@<1ex>[r] & B.}$$
\quad In parallel with single-valued functions, the composition and the cartesian product of two $m-$fibrations, and the $m-$pullback of an $m-$fibration are also $m-$fibration. 

\begin{definition}
	Let $\alpha : A \times B \rightrightarrows C$ be an $m-$continuous function. Then the function $\alpha_{adj} : A \rightrightarrows C^{B}$, $\alpha_{adj}(a)(D) = \alpha(a,b)$, is called an $m-$adjoint function for any subset $D \subseteq B$ containing $b$.
\end{definition}

\begin{proposition}
	$\alpha_{adj}$ is $m-$continuous provided that $\alpha : A \times B \rightrightarrows C$ is an $m-$continuous function. 
\end{proposition}

\begin{proof}
	Since $\alpha$ is upper semicontinuous (because it is $m-$continuous), we have that $\{(a,b) \in A \times B : \alpha(a,b) = \emptyset\}$ is open in $A \times B$ by Exercise 5.3.46 of \cite{Geletu:2006}. It follows that $\{a \in A : \alpha_{adj}(a)(D) = \emptyset\}$ is open in $A$ for any subset $D$ of $B$. $D$ contains $b$ (it is nonempty), so we obtain that $\{a \in A : \alpha_{adj}(a) = \emptyset\}$ is open in $A$. This means that $\alpha_{adj}$ is upper semicontinuous. Similarly, it can be easily shown that $\alpha_{adj}$ is lower semicontinuous by using Exercise 5.3.46 of \cite{Geletu:2006}.
\end{proof}

\quad Let $\beta : A \rightrightarrows C^{B}$ be an $m-$continuous function. Then the function
\begin{eqnarray*}
	\hat{\beta} : A \times B \rightrightarrows C
\end{eqnarray*}
defined by $\hat{\beta} = \text{Ev} \circ (\beta \times 1_{Y})$
is $m-$continuous since Ev, $\beta$, and the $m-$identical map $1_{Y}$ are all $m-$continuous functions.

\begin{theorem}
	The function $\gamma : C^{A \times B} \rightrightarrows C^{B^{A}}$, defined by $\gamma(\beta) = \beta_{adj}$, is an $m-$homeomorphism.
\end{theorem}

\begin{proof}
	Let $\gamma : C^{A \times B} \rightrightarrows C^{B^{A}}$ and $\eta : C^{B^{A}} \rightrightarrows C^{A \times B}$ be defined as $\gamma(\alpha) = \alpha_{adj}$ and $\eta(\beta) = \hat{\beta}$, respectively. Since $\gamma$ and $\eta$ are $m-$continuous, it is enough to show that the left (right) inverse of $\gamma$ is $\eta$, i.e., $\gamma$ is one-to-one (surjective).
	We obtain that
	\begin{eqnarray*}
		\gamma \circ \eta(\beta) = \gamma(\hat{\beta}) = \hat{\beta}_{adj} = \beta
	\end{eqnarray*}
    and
    \begin{eqnarray*}
    	\eta \circ \gamma(\alpha) = \eta(\alpha_{adj}) = \hat{\alpha_{adj}} = \alpha.
    \end{eqnarray*}
    Thus, $\gamma$ is bijective.
\end{proof}

\begin{remark}
	The diagram in Definition \ref{def3} is equivalent to its adjoint formulation
	\begin{displaymath}
		\xymatrix{W \ar@/_/[ddr]_\alpha \ar@/^/[drr]^\beta \ar@{.>}[dr]|-{\eta}\\
			&A^{I}_{m} \ar[d]^{\text{Ev}_{0}} \ar[r]_{\rho_{\ast}} & B^{I}_{m}\ar[d]_{\text{Ev}_{0}} \\
			&A \ar[r]^\rho &B.}
	\end{displaymath}
\end{remark}

\begin{theorem}\label{teo4}
	If $J : C \rightrightarrows A$ is an $m-$cofibration, then $J_{\ast} : B^{A} \rightrightarrows B^{C}$ is an $m-$fibration for any space $B$.
\end{theorem}

\begin{proof}
	Since $J$ is an $m-$cofibration, there exists an $m-$homotopy $\eta : W \times A \times I \rightrightarrows B$ such that the following diagram is commutative.
	\begin{displaymath}
		\xymatrix{W \times C \ar@{^{(}->}[rr] \ar[dd]_{1 \times J} & & W \times C \times I \ar[dd]^{1 \times J \times 1} \ar[dl]_{\beta} \\
			& B &\\
			W \times A \ar@{^{(}->}[rr] \ar[ur]^{\alpha} & & W \times A \times I \ar@{.>}[ul]_{\eta}}
	\end{displaymath}
    By adjunction, we have that the diagram
    \begin{displaymath}
    	\xymatrix{W \ar@/_/[ddr]_{\alpha^{'}} \ar@/^/[drr]^{\beta^{'}} \ar@{.>}[dr]|-{\eta^{'}}\\
    		&C^{A_{m}^{I}} \ar[d]^{\text{Ev}_{0}} \ar[r] & B^{C_{m}^{I}}\ar[d]_{\text{Ev}_{0}} \\
    		&B^{A} \ar[r]^{J_{\ast}} &B^{C}.}
    \end{displaymath}
    is commutative. The $m-$adjoint of $\eta$ is $\eta^{'}$, and thus, the $m-$hlp holds for $J_{\ast}$.
\end{proof}

\begin{example}
	Consider the pair $(\partial I,I)$. Since it has $m-$hep, $J : \partial I \rightrightarrows I$ is an $m-$cofibration. By Theorem \ref{teo4}, we observe that
	\begin{eqnarray*}
		&&\Pi: X^{I}_{m} \rightrightarrows X \times X \\ 
		&&\hspace*{0.8cm} \alpha \mapsto \Pi(\alpha) = \alpha(0) \times \alpha(1)
	\end{eqnarray*}
    is an $m-$fibration. 
\end{example}   

\section{Topological Multi-Complexity of Spaces}
\label{sec:3}

\quad Let $\rho : A \rightrightarrows B$ be an $m-$fibration. Then msecat$(\rho)$ of $\rho$ is the least positive integer $s$ provided that the following properties hold:
\begin{itemize}
	\item The open covering $\{C_{1},\cdots,C_{s}\}$ of $B$ exists.
	\item $\rho$ has an $m-$continuous $m-$section $\delta_{j} : C_{j} \rightrightarrows A$ on each $C_{j}$, $j \in \{1,\cdots,s\}$, i.e., \[\{b\} = \displaystyle \bigcap_{a \in \delta_{j}(b)}\rho(a)\] for all $b \in C_{j}$.
\end{itemize} 
In the case that such a covering does not exist, the msecat$(\rho)$ is $\infty$.  

\quad Let $X^{I}_{m}$ denote the space of all $m-$continuous paths in an $m-$pathwise connected space $X$. Define an $m-$fibration \[\Pi: X^{I}_{m} \rightrightarrows X \times X\] by $\Pi(\alpha) = \displaystyle \alpha(0) \times \alpha(1)$ for any $m-$path $\alpha$ in $X$. Then the topological multi-complexity of $X$ (denoted by tmc$(X)$) is the least positive integer $s$ provided that the following properties hold:
\begin{itemize}
	\item The open covering $\{D_{1},\cdots,D_{s}\}$ of $X \times X$ exists.
	\item $\Pi$ has an $m-$continuous $m-$section $\delta_{j} : D_{j} \rightrightarrows X^{I}_{m}$ on each $D_{j}$ with \linebreak$j \in \{1,\cdots,s\}$, i.e., \[\{(a,b)\} = \displaystyle \bigcap_{\alpha \in \delta_{j}(a,b)}\Pi(\alpha)\] for all $(a,b) \in D_{j}$. 
\end{itemize} 
In the case that such a covering does not exist, the topological multi-complexity is $\infty$. Alternatively, one has tmc$(X) =$ msecat$(\Pi)$.

\quad In this paper, we always consider $X$ as $m-$pathwise connected when we want to compute tmc$(X)$ unless otherwise stated.

\begin{definition}\label{def1}
	Assume that $\alpha$, $\beta : X \rightrightarrows Y$ are two $m-$continuous functions. Then an $m-$homotopic distance (multi-homotopic distance) between $\alpha$ and $\beta$ is the least positive integer $s$ provided that the following properties hold:
	\begin{itemize}
		\item The open covering $\{C_{1},\cdots,C_{s}\}$ of $X$ exists.
		\item For all $i \in \{1,\cdots,s\}$, $\alpha|_{C_{i}}$ is $m-$homotopic to $\beta|_{C_{i}}$.  
	\end{itemize}
    In the case that such a covering does not exist, the $m-$homotopic distance is $\infty$. 
\end{definition}

\quad The $m-$homotopic distance is denoted by D$^{m}(\alpha,\beta)$. The following results are easily obtained from Definition \ref{def1}:
\begin{itemize}
	\item[\textbf{i)}] The order of $\alpha$ and $\beta$ is not taken into account, i.e., D$^{m}(\alpha,\beta)$ = D$^{m}(\alpha,\beta)$.
	\item[\textbf{ii)}] That the $m-$homotopic distance between $\alpha$ and $\beta$ is $1$ means that $\alpha$ and $\beta$ are $m-$homotopic, i.e., D$^{m}(\alpha,\beta) = 1 \ \Leftrightarrow \ \alpha \simeq_{m} \beta$.
	\item[\textbf{iii)}] $\alpha_{1} \simeq_{m} \alpha_{2}$ and $\beta_{1} \simeq_{m} \beta_{2}$ imply that D$^{m}(\alpha_{1},\beta_{1}) =$ D$^{m}(\alpha_{2},\beta_{2})$.
	\item[\textbf{iv)}] D$^{m}(\eta \circ \alpha,\eta \circ \beta) \leq$ D$^{m}(\alpha,\beta)$ and D$^{m}(\alpha \circ \mu,\beta \circ \mu) \leq$ D$^{m}(\alpha,\beta)$ for any corresponding $m-$continuous functions.
\end{itemize}

\begin{theorem}\label{teo1}
	Assume that $\alpha$, $\beta: X \rightrightarrows Y$ are two $m-$continuous functions such that the diagram
	$$\xymatrix{
		P \ar@<1ex>[r]^{\Pi_{2}} \ar@<-1ex>[r] \ar@<1ex>[d] \ar@<-1ex>[d]_{PB} &
		Y^{I}_{m} \ar@<-1ex>[d] \ar@<1ex>[d]^{\Pi} \\
		X \ar@<-1ex>[r]_{(\alpha,\beta)} \ar@<1ex>[r] & Y \times Y}$$
	is commutative, where $PB$ is the $m-$pullback of $\Pi$ by $(\alpha,\beta)$. Then we have that D$^{m}(\alpha,\beta) =$ msecat$(PB)$.
\end{theorem}

\begin{proof}
	The proof is parallel with Theorem 2.7 of \cite{VirgosLois:2022}. However, there is only one point to note here: According to Mac\'{i}as-Virg\'{o}s and Mosquera-Lois, the homotopy $U \times I \rightarrow Y$ can be rewritten as the function $U \rightarrow Y^{I}$. This is also possible for $m-$functions because the compact-open topology on $Y^{I}_{m}$ is splitting (see \cite{GuptaSarma:2017} for more details on the compact-open and admissible topologies).
\end{proof}

\begin{corollary}
	For an $m-$projection $\rho_{i} : X \times X \rightrightarrows X$ with each $i \in \{1,2\}$, we have tmc$(X) =$ D$^{m}(\rho_{1},\rho_{2})$. 
\end{corollary}

\begin{proof}
	In Theorem \ref{teo1}, choose $\alpha = \rho_{1}$ and $\beta = \rho_{2}$ spesifically. Then we obtain $(\alpha,\beta)$ equals $1_{X \times X}$ and $PB = \Pi$.
\end{proof}

\begin{theorem}
	$X$ is an $m-$contractible space if and only if tmc$(X) = 1$.
\end{theorem}

\begin{proof}
	Let $X$ be $m-$contractible. Then the $m-$projection $\rho_{i} : X \times X \rightrightarrows X$ for $i \in \{1,2\}$ is null $m-$homotopic by Theorem 5 of \cite{KaracaOzkan:2023}. It follows that $$\text{D}^{m}(\rho_{1},\rho_{2}) = \text{D}^{m}(\zeta,\zeta),$$ where $\zeta$ is an $m-$constant map. Hence, tmc$(X) = 1$.
\end{proof}

\begin{example}
	$\mathbb{C}$ is $m-$contractible by considering an $m-$homotopy $\eta : \mathbb{C} \times I \rightrightarrows \mathbb{C}$ defined by $\eta(z,t) = \{(1-t)z + tz_{0}\}$ for a constant complex number $z_{0}$. Thus, tmc$(\mathbb{C}) = 1$.
\end{example}

\begin{figure}[h]
	\centering
	\includegraphics[width=0.50\textwidth]{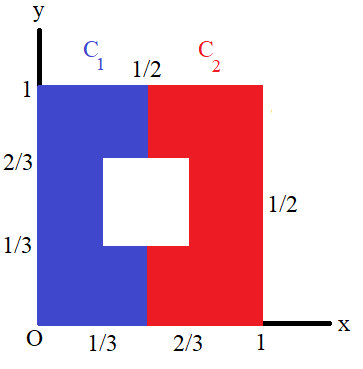}
	\caption{$m-$contractible subsets $C_{1}$ and $C_{2}$ of $X \times X$.}
	\label{fig:1}
\end{figure}

\begin{example}
	Consider the set $X = ([0,1] \times [0,1]) - ((1/3,2/3) \times (1/3,2/3))\subset \mathbb{R}^{2}$. Since it is not $m-$contractible, tmc$(X)$ is not $1$. $X \times X$ can be covered by two open sets
	\begin{eqnarray*}
		C_{1} = \{(x,y) \in X \times X : x, y \in ([0,1/2] \times [0,1]) - ((1/3,1/2) \times (1/3,2/3))\}
	\end{eqnarray*}
    and
    \begin{eqnarray*}
    	C_{2} = \{(x,y) \in X \times X : x, y \in ([1/2,1] \times [0,1]) - ((1/2,2/3) \times (1/3,2/3))\}.
    \end{eqnarray*}
    Moreover, for all $i \in \{1,\cdots,2\}$, one has \[\rho_{1}|_{C_{i}} \simeq_{m} \rho_{2}|_{C_{i}}\] by considering that $C_{1}$ and $C_{2}$ are $m-$contractible (see Figure \ref{fig:1}). This shows that tmc$(X) = 2$.
\end{example}

\quad By considering \textbf{iv)} (with \textbf{ii)}), if $\eta$ has a left $m-$homotopy inverse and $\nu$ has a right $m-$homotopy inverse, then the equalities \[\text{D}^{m}(\eta \circ \alpha,\eta \circ \beta) = \text{D}^{m}(\alpha,\beta)\] and \[\text{D}^{m}(\alpha \circ \nu,\beta \circ \nu) = \text{D}^{m}(\alpha,\beta)\] hold. Therefore, given $m-$homotopy equivalences $\eta_{1} : X_{1} \rightrightarrows Y_{1}$ and $\eta_{2} : X_{2} \rightrightarrows Y_{2}$ for which $\eta_{1} \circ \alpha \circ \eta_{2} \simeq_{m} \alpha^{'}$ and $\eta_{1} \circ \beta \circ \eta_{2} \simeq_{m} \beta^{'}$ for any functions $\alpha$, $\beta : Y_{2} \rightrightarrows X_{1}$ and $\alpha^{'}$, $\beta^{'} : X_{2} \rightrightarrows Y_{1}$, we have that \[\text{D}^{m}(\alpha,\beta) = \text{D}^{m}(\alpha^{'},\beta^{'}).\]  Consequently, we obtain an important property as follows.

\begin{theorem}\label{teo2}
   tmc$(X)$ is an $m-$homotopy invariant.
\end{theorem}

\section{Topological Multi-Complexity of Multi-Valued Fibrations}
\label{sec:4}

\begin{definition}\label{def2}
	Let $\alpha : X \rightrightarrows Y$ be a surjective $m-$fibration between two pathwise $m-$connected spaces $X$ and $Y$. Given an $m-$fibration $\Pi_{\alpha} : X^{I}_{m} \rightrightarrows X \times Y$ defined by $\Pi_{\alpha}(\rho) = \rho(0) \times (\alpha \circ \rho
	(1))$, the topological multi-complexity of $\alpha$ (denoted by tmc$(\alpha)$) is msecat($\Pi_{\alpha}$).
\end{definition}

\quad Note that $\Pi_{\alpha}$ is an $m-$fibration when $\alpha$ is an $m-$fibration because $\Pi_{\alpha}$ can be rewritten as $(1_{X} \times \alpha) \circ \Pi$. Indeed, $1_{X}$, $\alpha$, and $\Pi$ are all $m-$fibrations.

\quad By Theorem \ref{teo1}, we have the following result:

\begin{corollary}
	Let $\alpha : X \rightrightarrows Y$ be a surjective $m-$fibration between pathwise $m-$connected spaces $X$ and $Y$. Then tmc$(\alpha) =$ D$^{m}(\alpha \circ \rho_{1},\rho_{2})$ for the $m-$projection map $\rho_{i} : X \times X \rightrightarrows X$ with $i \in \{1,2\}$. 
\end{corollary}

\quad From Definition \ref{def2}, it is clear to have the following results as follows.
\begin{itemize}
	\item[\textbf{v)}] $\alpha = 1_{X} : X \rightrightarrows X$ implies that tmc$(1_{X}) =$ tmc$(X)$.
	\item[\textbf{vi)}] For any two $m-$homotopic $m-$fibrations $\alpha$, $\beta : X \rightrightarrows Y$, we have that tmc$(\alpha) =$ tmc$(\beta)$.
\end{itemize}

\begin{proposition}\label{prop1}
	Given a surjective $m-$fibration $\alpha : X \rightrightarrows Y$, tmc$(\alpha) = 1$ provided that $Y$ is $m-$contractible.
\end{proposition}

\begin{proof}
	Let $\rho_{i}$ be an $m-$projection on $X$ for each $i = 1,2$. If $Y$ is $m-$contractible (equivalently, $Y$ and $\{\{y_{0}\}\}$ has the same $m-$homotopy type for any fixed element $y_{0}$ in $Y$), then we have an $m-$function $\eta : Y \simeq_{m} \{\{y_{0}\}\}$ having the left and right homotopy inverse $\eta^{'} : \{\{y_{0}\}\} \simeq_{m} Y$. Therefore, we get
	\begin{eqnarray*}
		\text{tmc}(\alpha) = \text{D}^{m}(\alpha \circ \rho_{1},\rho_{2}) = \text{D}^{m}(\eta \circ \alpha \circ \rho_{1},\eta \circ \rho_{2}).
	\end{eqnarray*}
    Since $\eta \circ \alpha \circ \rho_{1}$ is $m-$homotopic to $\eta \circ \rho_{2}$, we obtain tmc$(\alpha) = 1$.
\end{proof}

\begin{example}
	Consider the bijective function $\alpha : \mathbb{C} \rightrightarrows \mathbb{C}$, $x \mapsto \{\sqrt{x},-\sqrt{x}\}$ having the inverse $\alpha^{-1} : \mathbb{C} \rightrightarrows \mathbb{C}$, $y \mapsto \alpha^{-1}(y) =
	\begin{cases}
		\{0\}, & \text{if } y = 0 \\
		\{y^2\}, & \text{if } y \neq 0
	\end{cases}$. We first show that $\alpha$ is upper semicontinuous. Let $V \subset \mathbb{C}$ be an open set in the codomain. For any $x \in \alpha^{+}(V)$, $\alpha(x) = \{\sqrt{x}, -\sqrt{x}\} \subset V$. Then, both $\sqrt{x}$ and $-\sqrt{x}$ are elements of $V$. Since $V$ is open, there exist positive real numbers $\epsilon_1$ and $\epsilon_2$ such that the open balls \[B(\sqrt{x}, \epsilon_1) = \{z \in \mathbb{C} : |z - \sqrt{x}| < \epsilon_1\}\] and \[B(-\sqrt{x}, \epsilon_2) = \{z \in \mathbb{C} : |z - (-\sqrt{x})| < \epsilon_2\}\] are contained in $V$. Now, we define $\epsilon = \min(\epsilon_1, \epsilon_2)$. We claim that the open ball $B(x, \epsilon) = \{y \in \mathbb{C} : |y - x| < \epsilon\}$ is contained within $\alpha^{+}(V)$. Consider any $y \in B(x, \epsilon)$. We obtain \[|y - \sqrt{x}| \leq |y - x| + |x - \sqrt{x}| < \epsilon + \epsilon_1 \leq \epsilon_1,\] and similarly, \[|y - (-\sqrt{x})| < \epsilon_2.\] Therefore, both $|y - \sqrt{x}|$ and $|y - (-\sqrt{x})|$ are less than their respective $\epsilon$ values, and hence, $\sqrt{y}$ and $-\sqrt{y}$ are elements of $V$. It follows that $\alpha(y) = \{\sqrt{y}, -\sqrt{y}\} \subset V$. For every $y \in B(x, \epsilon)$, $\alpha(y) \subset V$ because $y$ is chosen arbitrarily in $B(x, \epsilon)$. So, we get $B(x, \epsilon) \subset \alpha^{+}(V)$. Since $x$ is chosen arbitrarily in $\alpha^{+}(V)$, for every $x \in \alpha^{+}(V)$, there exists an open ball around $x$ that is contained within $\alpha^{+}(V)$. We find that $\alpha^{+}(V)$ is open, which means that $\alpha$ is upper semicontinuous. Similarly, it can be shown that $\alpha$ is lower semicontinuous. Thus, $\alpha$ is an $m-$continuous. Also, a similar process works for showing that $\alpha^{-1}$ is an $m-$continuous. Thus, by Proposition \ref{prop2}, we have that $\alpha$ is an $m-$fibration. In addition, $\mathbb{C}$ is $m-$contractible. So, Proposition \ref{prop1} concludes that tmc$(\alpha) = 1$.
\end{example}

\begin{remark}
	As another illustration, consider the $m-$function $\alpha : [0,1] \rightrightarrows [0,1]$, \[\alpha(x) = \begin{cases}
		\{\dfrac{1}{2}x\}, & x \in [0,\dfrac{1}{2}] \\
		\{1-\dfrac{1}{2}(1-x)\}, & x \in [\dfrac{1}{2},1],
	\end{cases}\] expressed in Example 4.1 of \cite{Anisiu:1981}. Since it is not lower semicontinuous (namely that it is not an $m-$continuous), $\alpha$ is not an $m-$fibration. This means that tmc$(\alpha)$ cannot be computed.
\end{remark} 

\begin{example}
	Consider the $m-$fibration $\alpha : S^1 \rightrightarrows S^1$ defined by $\alpha(x) = \{x, -x\}$ for each point $x \in S^1$. $\alpha$ is surjective since for any point $y \in S^1$, we can take $x = y$ or $x = -y$, and in both cases, we have $\alpha(x) = \{x, -x\} = \{y, -y\} = \{y\}$. Assume that $\rho_1, \rho_2 : S^1 \times S^1 \rightrightarrows S^1$ is the $m-$projection map, where $\rho_1(x, y) = x$ and $\rho_2(x, y) = y$. Then we have $\alpha \circ \rho_1(x, y) = \alpha(x) = \{x, -x\}$ and $\rho_2(x, y) = \{y\}$. The open covering $\{C_1, C_2\}$ of $S^1$ is defined as follows: $C_1 = S^1 - \{(0,1)\}$ and
	$C_2 = \{(0,1)\}$. Both $(\alpha \circ \rho_1)|_{C_j}$ and $\rho_2|_{C_j}$ are $m-$constant maps on $C_j$ with each $j = 1,2$. Hence, they are $m-$homotopic because $C_{1}$ and $C_{2}$ are $m-$contractible. Finally, we obtain that tmc$(\alpha) =2$.
\end{example}

\quad In parallel with tmc$(X)$, we have the quick important result for tmc$(\alpha)$:

\begin{theorem}\label{teo3}
	tmc$(\alpha)$ is a fiber $m-$homotopy equivalent invariant.
\end{theorem}

\section{The Lusternik-Schnirelmann Multi-Categories}
\label{sec:5}

\quad Let $X$ be a space. Then the Lusternik-Schnirelmann multi-category of $X$ (denoted by catm$(X)$) is the least positive integer $s$ provided that the following properties hold:
\begin{itemize}
	\item The open covering $\{D_{1},\cdots,D_{s}\}$ of $X$ exists.
	\item An $m-$inclusion $\nu_{j} : D_{j} \rightrightarrows X$ is null $m-$homotopic for each $j \in \{1,\cdots,s\}$.
\end{itemize} 
In the case that such a covering does not exist, the Lustenik-Schnirelmann multi-category is $\infty$. Alternatively, one has catm$(X) =$ D$^{m}(1_{X},\zeta_{x})$, where $\zeta_{x}$ is any $m-$constant map on $X$. Let $\alpha : X \rightrightarrows Y$ be an $m-$function. Then the Lusternik-Schnirelmann multi-category of $\alpha$ (denoted by catm$(\alpha)$) is the least positive integer $s$ provided that the following properties hold:
\begin{itemize}
	\item The open covering $\{D_{1},\cdots,D_{s}\}$ of $X$ exists.
	\item An $m-$inclusion function $\alpha|_{D_{j}} : D_{j} \rightrightarrows Y$ is null $m-$homotopic for each $j \in \{1,\cdots,s\}$.
\end{itemize} 
In the case that such a covering does not exist, the Lustenik-Schnirelmann $m-$category of $\alpha$ is $\infty$. Alternatively, one has catm$(\alpha) =$ D$^{m}(\alpha,\zeta)$, where $\zeta : X \rightrightarrows Y$ is any $m-$constant map. It is clear that $\alpha = 1_{X} : X \rightrightarrows X$ implies that catm$(1_{X}) =$ catm$(X)$. Furthermore, for any two $m-$homotopic $m-$fibrations $\alpha$, $\beta : X \rightrightarrows Y$, we have that catm$(\alpha) =$ catm$(\beta)$.

\begin{remark}
	catm$(X)$ is $1$ if and only if $X$ is $m-$contractible. Moreover, catm$(\alpha)$ is $1$ for an $m-$function $\alpha : X \rightrightarrows Y$ if and only if $Y$ is $m-$contractible.
\end{remark}

\begin{example}
	Consider two open hemispheres of the unit sphere $S^2$, say $D_1$ and $D_2$, obtained by cutting the sphere along a non-degenerate equatorial plane. These two hemispheres overlap along the equator, but each covers one hemisphere of the sphere. So, the covering is given by $\{D_1, D_2\}$. For each $j \in {1, 2}$, define the $m-$inclusion map $\nu_j: D_j \rightrightarrows S^2$ as follows. For $j = 1$, the function $\nu_1$ takes each point $x$ in the open hemisphere $D_1$ to itself on the sphere $S^2$. Formally, $\nu_1(x) = \{x\}$. For $j = 2$, the function $\nu_2$ takes each point $y$ in the open hemisphere $D_2$ to the antipodal point on the sphere $S^2$. Formally, $\nu_2(y) = \{-y\}$. Both of these functions have the property that they are null $m-$homotopic since we can continuously contract each point or its antipodal point to the basepoint on the sphere. Thus, catm$(S^2) = 2$.
\end{example}

\quad For the same reasons as Theorem \ref{teo2} and Theorem \ref{teo3}, we have the quick result:

\begin{theorem}
	\textbf{i)} catm$(X)$ is an $m-$homotopy invariant.
	
	\textbf{ii)} catm$(\alpha)$ is a fiber $m-$homotopy equivalent invariant.
\end{theorem}

\begin{proposition}
	\textbf{i)} Let $X$ be an $m-$pathwise connected space. Then we have catm$(X) \leq$ tmc$(X)$.
	
	\textbf{ii)} Let $X$ and $Y$ be any two $m-$pathwise connected spaces. Given a surjective $m-$fibration $\alpha : X \rightrightarrows Y$, we have that catm$(\alpha) \leq$ tmc$(\alpha)$.
\end{proposition}

\begin{proof}
	\textbf{i)} Let $\eta_{1}$, $\eta_{2} : X \rightrightarrows X \times X$ be $m-$functions defined as \[\eta_{1}(x) = \{(x,x_{0})\} \hspace*{0.6cm} \text{and} \hspace*{0.6cm} \eta_{2}(x) = \{(x_{0},x)\}\] for a fixed point $x_{0} \in X$. Then we obtain
	\begin{eqnarray*}
		\text{D}^{m}(1_{X},\zeta_{x}) = \text{D}^{m}(\rho_{2} \circ \eta_{2},\rho_{1} \circ \eta_{2}) \leq \text{D}^{m}(\rho_{2},\rho_{1}). 
	\end{eqnarray*}
	
	\textbf{ii)} Let $\eta_{1}$, $\eta_{2} : X \rightrightarrows X \times X$ be $m-$functions defined as \[\eta_{1}(x) = \{(x,x_{0})\} \hspace*{0.6cm} \text{and} \hspace*{0.6cm} \eta_{2}(x) = \{(x_{0},x)\}\] for a fixed point $x_{0} \in X$. Then we obtain
	\begin{eqnarray*}
		\text{D}^{m}(\alpha \circ \rho_{1},\rho_{2}) \geq \text{D}^{m}(\alpha \circ \rho_{1} \eta_{1},\rho_{2} \circ \eta_{1}) = \text{D}^{m}(\alpha,\zeta). 
	\end{eqnarray*}
\end{proof}

\begin{proposition}
	\textbf{i)} Let $X$ and $Y$ be two $m-$pathwise connected spaces. Given an $m-$function $\alpha : X \rightrightarrows Y$, we have that catm$(\alpha) \leq$ catm$(X)$.
	
	\textbf{ii)} Let $Y$ be an $m-$pathwise connected space. Given an $m-$function $\alpha : X \rightrightarrows Y$, we have that catm$(\alpha) \leq$ catm$(Y)$.
\end{proposition}

\begin{proof}
	\textbf{i)} Assume that $\alpha(x_{0}) = \{y_{0}\}$. Then we obtain
	\begin{eqnarray*}
		\text{D}^{m}(\alpha,\zeta) &\leq& \text{D}^{m}(\alpha \circ 1^{m}_{X},\zeta \circ 1^{m}_{X})
		=\text{D}^{m}(\alpha \circ 1^{m}_{X},\{y_{0}\}))
		=\text{D}^{m}(\alpha \circ 1^{m}_{X},\alpha(\{x_{0}\})) \\ &\leq& \text{D}^{m}(1^{m}_{X},\{x_{0}\}) = \text{D}(1^{m}_{X},\zeta_{x}),
	\end{eqnarray*}
	where $\zeta : X \rightrightarrows Y$ and $\zeta_{x} : X \rightrightarrows X$ are two $m-$constant maps given by $\zeta(x) = \{y_{0}\}$ and $\zeta_{x}(x) = \{x_{0}\}$, respectively.
	
	\textbf{ii)} Assume that $\zeta : X \rightrightarrows Y$ and $\zeta_{y} : Y \rightrightarrows Y$ are two $m-$constant maps given by $\zeta(x) = \{y_{0}\}$ and $\zeta_{y}(y) = \{y_{0}\}$, respectively. Then we obtain
	\begin{eqnarray*}
		\text{D}^{m}(1^{m}_{Y},\zeta_{y}) &\geq& \text{D}^{m}(1^{m}_{Y} \circ \alpha,\zeta_{y} \circ \alpha)
		=\text{D}^{m}(1^{m}_{Y} \circ \alpha,\{y_{0}\} \circ \alpha)) \\ 
		&=& \text{D}^{m}(\alpha,\zeta).
	\end{eqnarray*}
\end{proof}

\section{Conclusion}
\label{conc:5}

\quad Multi-valued functions stand as a testament to the evolving landscape of mathematical thought. Rooted in history, their properties and relationships with topology continue to inspire mathematicians, offering powerful tools for understanding complex systems and phenomena across various domains of mathematics and beyond. In this investigation, we add a new one to robotics applications by incorporating the topological complexity computation problem into multi-valued functions. With this study, which contributes to the robot motion planning problem, the first thing that can be done next is to introduce the higher topological $m-$complexity number. Along with the solution of this open problem, higher topological complexity numbers of an $m-$fibration can also be presented. Thus, less thought may be needed on the problem of controlling an autonomous robot in multiple environments.

\acknowledgment{The Scientific and Technological Research Council of Turkey TÜBİTAK-1002-A supported this investigation under project number 122F454.}

\end{document}